\documentclass[preprint,12pt]{elsarticle}

\usepackage{amsmath,amsthm,amscd,amssymb}
\usepackage{enumerate}
\usepackage{xcolor}
\usepackage{algorithm}
\usepackage{algpseudocode}

\newtheorem{theorem}{Theorem} 
\newtheorem{proposition}{Proposition}
\newtheorem{lemma}{Lemma}
\newtheorem{corollary}{Corollary}

\journal{Carpathian Mathematical Publications}

\begin{document}
	
	\begin{frontmatter}

		\title{Lie nilpotency index of skew symmetric elements in group algebras}

		\author{Zsolt Adam Balogh}
		\ead{baloghzsa@gmail.com}
		\fntext[label2]{\noindent The research was supported by UAEU UPAR Grant No. $G00004618$.\\
		\mbox{ }\quad\, Accepted for publication in Carpathian Mathematical Publications.}
		\address{Department of Mathematical Sciences\\
			 United Arab Emirates University, Al Ain, \\
			 United Arab Emirates, P.O.Box: 15551}

		\begin{abstract}
			Let $FG$ be the group algebra of a finite $p$-group $G$ over a finite field $F$ of characteristic $p$, where $p$ is an odd prime.  
			Let $*$ be the classical involution of $FG$ and let $FG^-$ be the Lie subalgebra of the skew symmetric elements with respect to $*$. 
			In this paper, we prove that the Lie nilpotency index of $FG$ is determined by its Lie subalgebra $FG^-$, in the case when $G$ is a $p$-group with a commutator subgroup of order $p$. The structure of the terms of the lower Lie central series of $FG^-$ has also been described.
		\end{abstract}
			
	\end{frontmatter}

\section{Introduction}

Let $FG$ be the group algebra of a finite $p$-group $G$ over a finite field $F$ of positive characteristic $p$. 
It is well-known that in $FG$, an involution can always be defined, namely the classical involution, which is a linear extension of the anti-automorphism of $G$ mapping each element to its inverse. The involutions of group algebras $FG$ and the specially constructed elements using involutions play an important role in examining the structure of the algebra. The examination of elements related to the classical involution of group algebras focuses on symmetric and skew-symmetric elements. An overview of this subject area is available in this book \cite{Lee2010}.

In the investigation of the unit group of group algebras, Laue \cite{Laue1984} and Du \cite{Du1992} have proven that there is a close relationship between the nilpotency class of the unit group and the Lie nilpotency index of the group algebra. This has provided additional impetus for the study of Lie properties, encompassing not only the Lie nilpotency index but also the study of the Lie derived length (\cite{Balogh_Juhasz_India,Balogh_Juhasz_Commalg_I, Lee1999,Sahai1995,Shalev1992,Shalev1994}) and the Lie Engel property (\cite{Artemovych2022,Balogh_PureAlgebra,Lee2000,Lee2015}). 
Several articles also deal with the determination of the Lie nilpotency index (\cite{Balogh2012}, \cite{Bovdi_2005}, \cite{Bovdi_2006}, \cite{Bodi_Spinelli_2004}, \cite{Bovdi_2010}, \cite{Sahai_2018}, \cite{Singh_2023}). The most comprehensive result in this field being achieved by Bhandari and Passi \cite{Bhandari1992}. 

It is well known that the set of skew symmetric elements $FG^-$ forms a Lie subalgebra in $FG$. However, little is known about the structure of this Lie subalgebra, for more details see also the book \cite{Lee2010}. In this paper we prove that in some cases the nilpotency index of the lie algebra of skew symmetric elements determines the nilpotency index of the entire algebra.

Let $L$ be a Lie algebra and let $\gamma_n(L)$ be the $n^{th}$ term of the lower Lie central series of $L$. Furthermore, let $t(L)$ denote the Lie-nilpotency index of $L$. We prove the following theorem and its corollaries.

\begin{theorem}\label{main_theorem}
	Let $G$ be a $p$-group with commutator subgroup of order $p$ and let $F$ be a field of characteristic $p$, where $p>2$. 
	Then $\gamma_n(FG^-)=\gamma_n(FG)\cap FG^- $.
\end{theorem}

Let $\varphi$ be a map on $FG$ defined by $\varphi(x)=x-x^*$. 

\begin{corollary}\label{cor:one} 
	Let $G$ be a $p$-group with commutator subgroup of order $p$ and let $F$ be a field of characteristic $p$, where $p>2$. 
	Then $\gamma_{n}(FG^-) = \varphi(\gamma_{n}(FG))$.
\end{corollary}

\begin{corollary}\label{cor:two}
	Let $G$ be a $p$-group with commutator subgroup of order $p$ and let $F$ be a field of characteristic $p$, where $p>2$.
	Then $t(FG)=t(FG^-)$.
\end{corollary}       

We should remark that if $p>3$ paper \cite{Bhandari1992} presents a formula for the Lie nilpotency index of $FG^-$.

\section{Lie nilpotency and skew elements}

The group algebra $FG$ can be regarded as an associated Lie algebra, via the Lie commutator $[x, y] = xy-yx$, where $x, y \in FG$. Set $[x_0, x_1,\cdots,x_n] =
[[x_0, x_1,\cdots, x_{n-1}], x_n]$, where
$x_0, x_1, \cdots x_n \in FG$. The $n^{th}$ term of the lower Lie central series $\gamma_n(FG)$
of $FG$ is the linear span generated by all Lie commutators $[x_0, x_1,\cdots, x_n]$,
where $\gamma_0(FG) = FG$. Similarly, if $S$ is a subset of $FG$, then $\gamma_n(S)=\underbrace{[S,S,\cdots,S]}_n$ is the linear span generated by all Lie commutators
$[x_0, x_1,\cdots, x_n]$, where $x_i \in S$ (see also \cite{Giambruno1993,Lee2010}). The set $S$ is called Lie nilpotent with nilpotency index $n$ if $\gamma_{n-1}(S)\ne 0$ and $\gamma_n(S)=0$. 

The following propositions are well-known.
\begin{proposition}[Theorem 3.1.1 in \cite{Lee2010}]
	Let $F$ be a field of characteristic $p \geq 0$ and let $G$ be a group. Then $FG$ is
	Lie nilpotent if and only if $G$ is nilpotent and $p$-abelian.
\end{proposition}

\begin{proposition}\cite{Giambruno1993}
	Let $G$ be a group with no $2$-elements. Let $F$ be a field and suppose that
	$FG^-$ is Lie nilpotent. Then 
	\begin{enumerate}
		\item If the characteristic of $F$ is zero, then $G$ is abelian.
		\item If the characteristic of $F$ is a prime $p$, then $G$ is nilpotent and $p$-abelian.
	\end{enumerate}	
\end{proposition}

We can conclude that
\begin{corollary}
	Let $F$ be a field of characteristic $p$, where $p$ is an odd prime and let $G$ be a $p$-group.
	Then $FG^-$ is Lie nilpotent if and only if $FG$ is Lie nilpotent.
\end{corollary}

Let us define the $n^{th}$ term of the upper Lie power series $FG^{(n)}$ of $FG$ as the associative ideal generated by $\gamma_n(FG)$. The $n^{th}$ Lie dimension subgroup of $FG$ is defined by $D_{(n)}(G)=G \cap (1 +FG^{(n)})$, where $n \geq 1$ (for more details, see the book by Passi \cite{Passi1979}).
Let $p^{d_{(m)}}$ defined by the index $[D(m)(G) : D(m+1)(G)]$.
Passi and Bhandari presented a formula for the Lie-nilpotency class of $FG$ when $p>3$.
\begin{proposition}\cite{Bhandari1992}\label{theorem_Bandhari}
	Let $F$ be a field of characteristic $p > 3$ and let $G$ be a group such
	that $FG$ is Lie-nilpotent. Then
	$t(FG) = 2 + (p-1) \sum_{m \geq 1} md_{m+1}$.
\end{proposition}

The following well-known identity will be helpful in the proof of the main theorem
\begin{equation}\label{eq:liecomm}
	[xy,z]=x[y,z]+[x,z]y.
\end{equation}

Let $\varphi$ be a map on $FG$ defined by $\varphi(x)=x-x^*$. Evidently, $\varphi : FG \rightarrow FG^-$ and the restriction of $\varphi$ on $FG^-$ $\varphi|_{FG^-}$ is a one-to-one map.
\begin{lemma}\label{lemma:1}
	Let $G$ be a group and let $F$ be a field of characteristic $p$. Then
	$\varphi$ is a homomorphism on $(FG,+)$ with kernel $FG^+$.
	Furthermore, if $p>2$, then $\varphi|_{FG^-}$ is an automorphism on $(FG^-,+)$.
\end{lemma}
\begin{proof}
	Obviously,
	$\varphi(x+y)=x+y-(x+y)^*=x-x^*+y-y^*=\varphi(x)+\varphi(y)$.
	Since the set $\{ g-g^* | g\in G\}$ is a basis for $FG^-$ over $F$ and $\varphi(g-g^*)=2(g-g^*)$ we conclude that $\varphi(x)=2x$ for every $x\in FG^-$. Therefore, $\varphi$ is an automorphism when $p$ is odd.
\end{proof}

\begin{lemma}\label{lemma:2}
	If $z\in \gamma_n(FG)$, then $z^*\in \gamma_n(FG)$.
\end{lemma}
\begin{proof}
	Assume that $x,y\in FG$. Then
	\begin{equation}\label{eq:1}
		[x,y]^*=(xy-yx)^*=y^*x^*-x^*y^*=[y^*,x^*]=-[x^*,y^*].
	\end{equation}
    Therefore $[x,y]^*\in \gamma_2(FG)$.
	Assume that $x\in FG$ and $y\in \gamma_{n-1}(FG)$ for some $n$. By induction, $y^*\in \gamma_{n-1}(FG)$, so 
    $[x,y]^*=-[x^*,y^*]\in \gamma_{n}(FG)$.
\end{proof}
    
In the following lemmas and corollaries the conjugation $h^{-1}gh$ is abbreviated as $g^h$ and the commutator $g^{-1}h^{-1}gh$ as $(g,h)$.
\begin{lemma}\label{lemma}
	Let $G$ be a $p$-group with commutator subgroup of order $p$, where $p$ is an odd prime.
	For every $g,h\in G$ there exists an element $z_{h,g}\in FG$ satisfying $\varphi([h,g])=[z_{h,g}-z_{h,g}^*,g-g^*]$.
\end{lemma}    
    
\begin{proof}
	Suppose that there exists $z_{h,g}\in FG$ such that $\varphi([h,g])=[z_{h,g}-z_{h,g}^*,g-g^*]$. 
	Then
	\begin{align*}
		\varphi([h,g])=&[z_{h,g}-z_{h,g}^*,g-g^*]=\varphi([z_{h,g},g])-\varphi([z_{h,g},g^*])=\\
		&\varphi([h,g])+\varphi([z_{h,g},g])-\varphi([h,g])-\varphi([z_{h,g},g^*])=\\
		&\varphi([h,g])+\varphi([z_{h,g}-h,g])-\varphi([z_{h,g},g^*]).
	\end{align*}
	Therefore it is enough to prove that there exists $z_{h,g}\in FG$ such that 
	$\varphi([z_{h,g}-h,g])=\varphi([z_{h,g},g^*]$.
	Suppose that $z_{h,g}=xh$ for some $x\in F(C_G(g))$, where $C_G(g)$ is the centralizer of $g$ in $G$. Then
	\begin{align*}
		\varphi([(x-1)h,g])=&\varphi([xh,g^*].
	\end{align*}
	By calculating the left and right sides of the equation independently, we get
	\begin{align*}
		&\varphi([(x-1)h,g])=([(x-1)h,g])-([(x-1)h,g])^*=\\
		&([(x-1)h,g])+([h^*(x^*-1),g^*])=\\
		&(x-1)hg-g(x-1)h+h^*(x^*-1)g^*-g^*h^*(x^*-1)=\\
		&xg^{h^*}h-g^{h^*}h-xgh+gh+(x^*)^h(g^*)^hh^*-(g^*)^hh^*-g^*(x^*)^hh^*+g^*h^*
	\end{align*}
	and
	\begin{align*}
		&\varphi([xh,g^*]=[xh,g^*]-[xh,g^*]^*=(xh)g^*-g^*(xh)+(h^*x^*)g-g(h^*x^*)=\\
		&x(g^*)^{h^*}h+x(g^*)^hh^*-g^*xh+(x^*g)^hh^*-g(x^*)^hh^*
	\end{align*}
    Comparing the supports of the obtained expressions we have that 
    \[
	xg^{h^*}-g^{h^*}-xg+g=x(g^*)^{h^*}-g^*x
	\] and 
	\[
	(x^*)^h(g^*)^h-(g^*)^h-g^*(x^*)^h+g^*=(x^*g)^h-g(x^*)^h.
	\]

    Therefore
    \begin{align*}
	&g-g^h=x^h(g-g^h-(g-g^h)^*).
    \end{align*}
    Let $x=h[1+g^{-2}(g,h)^{-1}]^{-1}h^{-1}$ be, where $(g,h)=g^{-1}h^{-1}gh$. Then 
    \begin{align*}
	&[1+g^{-2}(g,h)^{-1}](g-g^h)=g-g^h+g^{-2}(g,h)^{-1}g-g^{-2}(g,h)^{-1}g^h=\\
	&=g-g^h+h^{-1}g^{-1}h-g^{-1}=g-g^h+(g^*)^h-g^{*}=g-g^h-(g-g^h)^*,
    \end{align*}
    so $z_{h,g}=h[1+g^{-2}(g,h)^{-1}]^{-1}$, and the proof is complete.
\end{proof}    

\begin{corollary}\label{cor:2}
	Let $G$ be a $p$-group with commutator subgroup of order $p$, where $p$ is an odd prime.
	For every $x\in FG$ and $g\in G$ there exists an element $z_{x,g}\in FG$ satisfying $\varphi([x,g])=[z_{x,g}-z_{x,g}^*,g-g^*]$.
\end{corollary}
\begin{proof}
	Assume that $\sum_{h\in G}\alpha_h h\in FG$.
	\begin{align*}
		\varphi([x,g])=&\varphi([\sum_{h\in G}\alpha_h h,g])=\sum_{h\in G}\alpha_h \varphi([h,g]).
	\end{align*}	
	According to Lemma \ref{lemma} $\varphi([h,g])=[z_{h,g}-z_{h,g}^*,g-g^*]$ for some $z_{h,g}\in FG$.
	Therefore
	\[
	\varphi([x,g])=[\sum_{h\in G}\alpha_h (z_{h,g}-z_{h,g}^*),g-g^*]=[\sum_{h\in G}\alpha_h z_{h,g}-\big(\sum_{h\in G}\alpha_h  z_{h,g}\big)^*,g-g^*].
	\]

\end{proof}

Let $C_g$ denote the conjugacy class of $g$ in $G$ and we will denote by $G'$ the commutator subgroup of $G$.
\begin{lemma}\label{lemma:4}
	Let $G$ be a $p$-group with $|G'|=p$. Then
	$\gamma_n(FG)=\gamma_n(FG)FG=FG\gamma_n(FG)$.
\end{lemma}
\begin{proof}
	Let us prove first that $\gamma_1(FG)=\gamma_1(FG)FG$. Evidently,  $\gamma_1(FG)$ is a subset of $\gamma_1(FG)FG$, so it is enough to prove that $\gamma_1(FG)FG \subseteq \gamma_1(FG)$.
	
	Let $g_1,g_2,h$ be the elements of $G$. Then 
	\begin{align*}
	[g_1,g_2]h=&g_1g_2h-g_2g_1h=g_1g_2h-g_1g_2(g_2,g_1)h=g_1g_2h-g_1g_2h(g_2,g_1).
	\end{align*}
    Since $|G'|=p$ every conjugacy class of a non central element $g \in G$ can be written as $C_g=gG'$.
    Therefore $C_{g_1g_2h}=g_1g_2hG'$ and we conclude that $g_1g_2h$ and $g_2g_1h$ are conjugates. Thus
    $g_2g_1h=t^{-1}g_1g_2ht$ for some $t\in G$. Furthermore, $g_1g_2h-t^{-1}g_1g_2ht=[g_1g_2ht,t^{-1}]$, which proves that $[g_1,g_2]h\in \gamma_1(FG)$. Similar consideration shows that $\gamma_1(FG)=FG\gamma_1(FG)$.
	 
	Suppose that $\gamma_{n-1}(FG)=\gamma_{n-1}(FG)FG=FG\gamma_{n-1}(FG)$ for some $n$. Let $x$ be an element of $\gamma_{n-1}(FG)$ and $y_1,y_2$ are elements of $FG$.
	According to Eq. \ref{eq:liecomm}
	\[ [x,y_1]y_2=[xy_2,y_1]-x[y_2,y_1].\]
	The inductive hypothesis forces that $xy_2 \in \gamma_{n-1}(FG)$ so $[xy_2,y_1]\in \gamma_{n}(FG)$. Also due to the inductive hypothesis it follows that $x[y_2,y_1] \in \gamma_{n-1}(FG) \subset \gamma_{n}(FG)$ and so $[x,y_1]y_2 \in \gamma_{n}(FG)$. This proves that $\gamma_n(FG)=\gamma_n(FG)FG$.
	
	Similarly, we can prove that $\gamma_{n}(FG)=FG\gamma_{n}(FG)$.
 
\end{proof}    
    
Let $S_n$ be the symmetric group of degree $n$ and let $FS_n$ be its group algebra over the field $F$. 
For every $x\in FS_n$ we define a map $f_x : G^n \rightarrow FG$, where $G^n=\{(g_1,g_2,\cdots,g_n)| g_i \in G\}$.
Let $x = \sum_{\sigma \in S_n}\alpha_{\sigma} \sigma \in FS_n$ be and 
\[
f_x\big((g_1,g_2,\cdots,g_n)\big)=\sum_{\sigma \in S_n}\alpha_{\sigma} \big(g_{\sigma(1)}g_{\sigma(2)}\cdots g_{\sigma(n)}\big)\in FG.
\]

Let $x=\sum_{g\in G} \alpha_g g$ and $y=\sum_{h\in G} \beta_h h$ be elements of $FG$. Evidently, $[x,y]=\sum_{g,h\in G}\alpha_g\beta_h[g,h]$ and the set $S_1=\{[g,h]|g,h\in G\}$  spans $\gamma_{1}(FG)$. Choose a basis $B_{1}$ for $\gamma_{1}(FG)$ such that $B_{1}$ is a subset of $S_1$.
Using induction we can define a basis $B_n$ on $\gamma_{n}(FG)$ such that every element of $B_n$ has the form $[g_0, g_1,\cdots,g_n]$, where $g_i\in G$. Let us call this type of basis a simple basis of $\gamma_{n}(FG)$.

Since $B_0^-=\{g-g^*| g\in G \}$ forms a basis of $FG^-$ we can define a basis $B_n^-$ of $\gamma_{n}(FG)^-$ such that every element of $B_n^-$ has the form $[g_0-g_0^*, g_1-g_1^*,\cdots,g_n-g_n^*]$, where $g_i\in G$. Such a basis is also called a simple basis of $\gamma_{n}(FG^-)$.

\begin{lemma}\label{lemma:x}
	Let $G$ be a $p$-group.
	Every element of a simple basis $B_n$ of $FG$ can be written as $f_{x_n}\big((g_1,g_2,\cdots,g_n)\big)$ for some $x_n \in FS_n$ and $g_i \in G$.
\end{lemma}
\begin{proof}
	Let $x_2=(1,2)-(2,1)$ be an element of $FS_2$. Then 
	\[[g_1,g_2]=g_1g_2-g_2g_1=f_{x_2}\big((g_1,g_2)\big).\] 
	Suppose that 
	$[g_1,g_2,\cdots,g_{n-1}]=f_{x_{n-1}}\big((g_1,g_2,\cdots,g_{n-1})\big)$ for some $x_{n-1}\in FS_{n-1}$. Then
	\begin{align*}
	[g_1,g_2,\cdots,g_{n}]=&f_{x_{n-1}}\big((g_1,g_2,\cdots,g_{n-1})\big)g_n-g_nf_{x_{n-1}}\big((g_1,g_2,\cdots,g_{n-1})\big)=\\
	&f_{x_{n}}\big((g_1,g_2,\cdots,g_n)\big)
	\end{align*}
	for some $x_n$ in $FS_n$. The proof is then enforced by the definition of $B_n$.
\end{proof}

\begin{corollary}\label{cor:basis}
	Let $G$ be a $p$-group.
	For every element of a simple basis $b_n\in B_n$ the support of $b_n$ contains multiplications in the form $g_{\sigma{(1)}}g_{\sigma{(2)}}\cdots g_{\sigma{(n)}}$ for some $\sigma \in S_n$, where $g_1,g_2,\cdots , g_n$ are fixed elements of $G$.	
\end{corollary}    

\begin{lemma}\label{lemma:5}
	Let $G$ be a $p$-group with $|G'|=p$ and $p>2$.
	For every $b_n\in B_n$ and $g\in G$ there exists an element $z_{b_n,g}\in \gamma_n(FG)$ satisfying $\varphi([b_n,g])=[z_{b_n,g}-z_{b_n,g}^*,g-g^*]$.
\end{lemma}
\begin{proof}
	The sentence is true for $n=0$ by Corollary \ref{cor:2}. Suppose it is true for $k<n$.
	Let $b_n$ be an element of the simple basis $B_n$. Evidently, $b_n=[b_{n-1},h_n]$ for some $b_{n-1}=\sum_{h\in G}\alpha_h h \in B_{n-1}$ and $h_n\in G$.
	
	For an $x=\sum_{s\in G} \alpha_s s\in FG$ and $g\in G$ let us define $z_{x,g}$ by the formula
	\[
	z_{x,g}=\sum_{s\in G} \alpha_s \big(s(1+g^{-2}(g,s))^{-1}\big).
	\] 

    Keeping in mind that $b_{n-1}=\sum_{h\in G}\alpha_h h$, we have that $b_{n-1}h_n=\sum_{h\in G}\alpha_h hh_n$ and $x^{h_n^*}=h_nb_{n-1}=\sum_{h\in G}\alpha_h h_nh$. Therefore
    \[
	z_{b_n,g}-z_{b_n^{h_n^*},g}=  z_{b_{n-1}h_n,g}-z_{h_nb_{n-1},g}=    
    \]
	\begin{align*}
	&\sum_{h\in G}\alpha_h hh_n(1+g^{-2}(g,hh_n)^{-1})^{-1}-\sum_{h\in G}\alpha_h h_nh(1+g^{-2}(g,h_nh)^{-1})^{-1}=\\
	&\sum_{h\in G}\alpha_h hh_n(1+g^{-2}(g,h)^{-1}(g,h_n)^{-1})^{-1}-\sum_{h\in G}\alpha_h h_nh(1+g^{-2}(g,h_n)^{-1}(g,h)^{-1})^{-1}=\\
	&\sum_{h\in G}\alpha_h [h,h_n](1+g^{-2}(g,h)^{-1}(g,h_n)^{-1})^{-1}.
	\end{align*}	
    Since the nilpotency class of $G$ is two and Corollary \ref{cor:basis} states that $h=h_{\sigma(1)}h_{\sigma(2)}\cdots h_{\sigma(n-1)}$ for some $\sigma\in S_{n-1}$ and $h_i\in G$ we conclude that
    \begin{align*}
    (g,h_{\sigma(1)}h_{\sigma(2)}\cdots h_{\sigma(n-1)})=&(g,h_{\sigma(1)})(g,h_{\sigma(2)})\cdots (g,h_{\sigma(n-1)}) =\\
    &(g,h_1)(g,h_2)\cdots (g,h_{n-1})=(g,h_1h_2\cdots h_{n-1}).
    \end{align*}
	Therefore $z_{b_{n-1}h_n,g}-z_{h_nb_{n-1},g}=$
	\begin{align*}
    &\sum_{h\in G}\alpha_h [h,h_n](1+g^{-2}(g,h_1h_2\cdots h_{n-1}h_n)^{-1})^{-1}=\\
	&[\sum_{h\in G}\alpha_h h,h_n](1+g^{-2}(g,h_1h_2\cdots h_{n-1}h_n)^{-1})^{-1}=\\
	&b_n(1+g^{-2}(g,h_1h_2\cdots h_{n-1}h_n)^{-1})^{-1}.
	\end{align*}
    We define $z_{b_n,g}$ to be $z_{b_{n-1}h_n,g}-z_{h_nb_{n-1},g}$. 
    Since $b_n\in \gamma_{n}(FG)$ Lemma \ref{lemma:4} implies that $z_{b_n,g}\in \gamma_{n}(FG)$.
\end{proof}
    
\begin{corollary}\label{cor:3}
	Let $G$ be a $p$-group with $|G'|=p$ and $p>2$.
	For every $x\in \gamma_n(FG)$ and $g\in G$ there exists an element $z_{x,g}\in \gamma_n(FG)$ such that $\varphi([x,g])=[\varphi(z_{x,g}),\varphi(g)]$.
\end{corollary}    
  
Now, we are ready to prove the main theorem.  

\begin{proof}[Proof of the main theorem]
	Evidently, if $x\in \gamma_n(FG^-)$, then $x$ is skew symmetric and $x\in \gamma_n(FG)$. Therefore
	$\gamma_n(FG^-) \subseteq \gamma_n(FG)\cap FG^-$.
	
	Now, let us prove that if $x\in \gamma_n(FG)\cap FG^-$, then $x\in \gamma_n(FG^-)$.
	Evidently, $\gamma_0(FG)\cap FG^-=FG^-$, so the sentence is true for $n=0$.

	Assume that $n>0$ and $x\in \gamma_n(FG)\cap FG^-$. Since every $x\in \gamma_n(FG)$ can be written as $x=\sum_{g \in G}\sum_{b\in B_{n-1}} \alpha_{b,g} [b,g]$, where $\alpha_{b,g}\in F$ we conclude that
	\[
	\varphi(x)=\sum_{g \in G}\sum_{b\in B_{n-1}} \alpha_{b,g}\varphi([b,g]).
	\]
	According to Lemma \ref{lemma:5} there exist elements $z_{b,g}\in \gamma_{n-1}(FG)$ such that
	\[
	\varphi(x)=\sum_{g \in G}\sum_{b\in B_{n-1}} \alpha_{b,g} [z_{b,g}-z_{b,g}^*,g-g^*].
	\]
	By Lemma \ref{lemma:2} $z_{b,g}-z_{b,g}^* \in \gamma_{n-1}(FG^-)$.
	Since $x$ is skew symmetric Lemma \ref{lemma:1} states that $x=2^{-1}\varphi(x)\in \gamma_{n}(FG^-)$ and the proof is completed.
\end{proof}

We should remark that the relation $\gamma_n(FG)\cap FG^- \subseteq \gamma_n(FG^-)$ is not true in general. For example for the group
algebra $FG$, where $G$ is identified in the library of small groups of GAP \cite{GAP4} as [243, 3] and $F$ is the field of three elements 
$\gamma_3(FG)\cap FG^-$ is not a subset of $\gamma_3(FG^-)$.	

\begin{proof}[Proof of Corollary $1$]
	It follows immediately from Theorem \ref{main_theorem}, Lemma \ref{lemma:1} and Lemma \ref{lemma:2}.
\end{proof}

\begin{proof}[Proof of Corollary $2$]
	It is also a direct consequence of Theorem \ref{main_theorem}. 
\end{proof}

\end{document}